\documentclass{birkjour}
\usepackage{graphicx}
\usepackage{amssymb}
\usepackage{amsthm}
\usepackage{color}
\newfont{\bbb}{msbm10 scaled\magstep 1}

\newtheorem{thm}{Theorem}

\newtheorem{lem}{Lemma}

\newtheorem{pro}{Proposition}

\begin{document}

\title{Covering a reduced spherical body by a disk}

\author{ Micha{\l} Musielak}

\address{
University of Science and Technology\\
Kaliskiego 7, 85-789, Bydgoszcz, Poland}

\email{michal.musielak@utp.edu.pl}



\date{}

\begin{abstract}
\noindent
In this paper, the following two theorems are proved:  $(1)$ every spherical convex body $W$ of constant width $\Delta (W) \geq \frac{\pi}{2}$ may be covered by a disk of radius $\Delta(W) + \arcsin \left( \frac{2\sqrt{3}}{3} \cdot \cos \frac{\Delta(W)}{2}\right) - \frac{\pi}{2}$; $(2)$ every  reduced spherical convex body $R$ of thickness $\Delta(R)<\frac{\pi}{2}$  may be covered by a disk of radius $\arctan \left( \sqrt{2} \cdot \tan \frac{\Delta(R)}{2}\right)$.  
\end{abstract}

\maketitle

\section{Introduction}

The subject of reduced bodies in the Euclidean space $E^d$ has been well researched, see for instance \cite{LM}. 
However, it is a natural idea to investigate such bodies in non-Euclidean geometries. 
There are number of articles describing reduced bodies on the sphere, see: \cite{L2}, \cite{HN}, \cite{LMu}, \cite{LMu2}. 
The question arises which of the results achieved in the Euclidean space can be transferred to the sphere. 
In particular, the main theorem of this paper presents the spherical version of the variant of Jung theorem for reduced bodies in $E^2$ given in \cite{L3}.

\smallskip

Let $S^2$ be the unit sphere of the three-dimensional Euclidean space $E^{3}$. 
By a \emph{great circle} of $S^2$ we mean the intersection of $S^2$ with any two-dimensional subspace of $E^{3}$. 
The set of points of a great circle of $S^2$ in the distance at most $\frac{\pi}{2}$ from a point $c$ of this great circle is called a \emph{semicircle} with \emph{center} $c$. 
Any pair of points obtained as the intersection of $S^2$ with a one-dimensional subspace of $E^{3}$ is called a pair of \emph{antipodes}.
Note that if two different points $a,b$ are not antipodes, there is exactly one great circle containing them.
The shorter part of this great circle is called the \emph{spherical arc connecting $a$ and $b$}, or shortly \emph{arc}.  
It is denoted $ab$.
By the \emph{spherical distance} $|ab|$, or shortly \emph{distance}, of 
these points we mean the length of the arc $ab$.
If $a,b$ are antipodes, we put $|ab|=\pi$. 
If $p$ is a point of $S^2$ and $F$ is a closed set containing at least two points, then we define $\textrm{dist} (p, F)$ as $\min_{q\in F}  |pq|$.

A subset of $S^2$ is called \emph{convex} if it does not contain any pair of antipodes of $S^2$ and if together with every two points it contains the arc connecting them.
By a \emph{spherical convex body} we mean a closed convex set with non-empty interior.
If there is no arc in the boundary of a spherical convex body, we call the body \emph{strictly convex}.

Let $\rho \in \left(0, \frac{\pi}{2}\right]$.
By \emph{disk} of \emph{radius} $\rho$ and \emph{center} $c$ we mean the set of points of $S^2$ in the distance at most $\rho$ from a point $c \in S^2$.
The boundary of a disk is called a \emph{spherical circle}. 
By a \emph{hemisphere} we mean any disk of radius $\frac{\pi}{2}$.
The hemisphere with center $p$ is denoted by $H(p)$.
If $p,q$ is are antipodes, then $H(p)$ and $H(q)$ are called \emph{opposite hemispheres}.
Let $t$ be a boundary point of a convex body $C \subset S^2$.
We say that a hemisphere $H$ \emph{supports} $C$ at $t$ if $C \subset H$ and $t$ belongs to the great circle bounding $H$.
If the body $C$ is supported at $p$ by exactly one hemisphere, we call $p$ \emph{a smooth point} of the boundary of $C$. 
If all boundary points of $C$ are smooth, we say that $C$ is \emph{smooth}.
We say that $e$ is an \emph{extreme point} of $C$ if $C\setminus \{ e\}$ is a convex set.
If the set $A$ is contained in an open hemisphere, we define $\textrm{conv} (A) $ as a smallest convex set containing $A$ (for details see definition before Lemma 1 in \cite{L2})
 
If hemispheres $G$ and $H$ are different and not opposite, then $L = G \cap H$ is called a \emph{lune}. 
The two semicircles bounding $L$ and contained in $G$ and $H$, respectively, are denoted by $G/H$ and $H/G$. 
The \emph{thickness} $\Delta (L)$ of $L \subset S^2$ is defined as the distance of the centers of $G/H$ and $H/G$.  
By the \emph{corners} of $L$ we understand the two points of the set $(G/H )\cap (H/G)$.

For every hemisphere $K$ supporting a convex body $C\subset S^2$ we find hemispheres $K^*$ supporting $C$ such that the lunes $K\cap K^*$ are of the minimum thickness (by compactness arguments at least one such a hemisphere $K^*$ exists).
The thickness of the lune $K \cap K^*$ is called \emph{the width of $C$ determined by} $K$ and it is denoted by ${\rm width}_K (C)$ (see \cite{L2}).
If for all hemispheres $K$ supporting $C$ the numbers ${\rm width}_K (C)$ are equal, we say that $C$ is {\it of constant width} (see \cite{L2}).
By the \emph{thickness} $\Delta (C)$ of a convex body $C \subset S^2$ we understand the minimum width of $C$ determined by $K$ over all supporting hemispheres $K$ of $C$ (see \cite{L2}).

After \cite{L2} we call a spherical convex body $R \subset S^2$ \emph{reduced} if $\Delta (Z) < \Delta (R)$ for every convex body $Z \subset R$ different from $R$. 
Simple examples of reduced spherical convex bodies on $S^2$ are spherical bodies of constant width and, in particular, the disks on $S^2$.
Also each of the four parts of a spherical disk  on $S^2$ dissected by two orthogonal great circles through the center of the disk is a reduced spherical body.
It is called a \emph{quarter of a spherical disk}.

\medskip

Note: most of the above notions can be defined also in higher dimensions, for details see for instance \cite{L2}. 

\medskip

For the convenience of the reader recall a few formulas of spherical geometry which are frequently applied in this paper. 
Consider the right spherical triangle with hypotenuse $C$ and legs $A,B$. 
Denote by $\alpha$ the angle opposite to $A$ and by $\beta$ the angle opposite to  $B$.
By \cite{Mur} the following formulas hold true
\begin{equation}\label{cos}
\tan A = \cos \beta \tan C
\end{equation}
\begin{equation}\label{sin}
\sin A = \sin \alpha \sin C
\end{equation}
\begin{equation}\label{Pyth}
\cos C = \cos A \cos B
\end{equation}
\begin{equation}\label{cot}
\cos C = \cot \alpha \cot \beta  .
\end{equation} 

\medskip

After \cite{De} we define the \emph{circumradius} of a convex body $C$ as the smallest $\rho$ such that $C$ can be covered by a disk of radius $\rho$.
Such a disk is unique and is called the \emph{disk circumscribed about $C$}. 
The boundary of this disk is called the \emph{circle circumscribed about $C$}. 
In particular the circle circumsribed on a spherical triangle contains all vertices of this triangle, which is a consequence of the spherical Ceva's theorem (see: Theorem 2 of \cite{Ma}).

\smallskip

Determining the circumradius in the three special cases presented in the below Lemmas \ref{quart}-\ref{equi} is useful in the further part of the paper.

\begin{lem}\label{quart}
Let $Q \subset S^2$ be a quarter of disk. The circumradius of $Q$ is equal $\rho = \arctan \left( \sqrt{2} \cdot \tan \frac{\Delta(Q)}{2}\right)$.
\end{lem}

\begin{proof}
Denote by $c$ the center of the disk whose quarter is considered and by $a,b$ the two different extreme points of $Q$ such that $ca$ and $cb$ are subsets of the boundary of $Q$. 
It is easily seen that if a disk contains points $a,b,c$, then it contains $Q$.
Therefore we are looking for the radius $\rho$ of the disk circumscribed on the triangle $abc$.
Denote by $o$ the center of this disk and by $p$ a point on $ac$ such that the angle $\angle cpo$ is right.
Since $o$ is in equal distances from $a$ and $c$, the point $p$ is in the middle of $ac$ and thus $|cp|=\frac{\Delta (Q)}{2}$.
Clearly, the angle $\angle pco$ equals $\frac{\pi}{4}$ and $|oc|$ is equal to the radius of our disk.
Hence, considering the triangle $cop$, by (\ref{cos}) we have $\tan \frac{\Delta(Q)}{2}= \cos \frac{\pi}{4} \tan \rho$.
By evaluating $\rho$ we obtain the thesis of our lemma.
\end{proof}

Recall after \cite{L2} that \emph{Reuleaux triangle} is the intersection of three disks of radius $\sigma$ such that the centers of these disks are pairwise distant by $\sigma$.

\begin{lem}\label{Reul}
The circumradius of a spherical Reuleaux triangle $R$ is equal \\
$\rho= \arcsin \left( \frac{2\sqrt{3}}{3} \cdot \sin \frac{\Delta(R)}{2}\right)$. 
\end{lem}
\begin{proof}
Denote by $a,b,c$ the three points of $\textrm{bd}(R)$ that are not smooth. 
It is easy to check that any disk containing these three points contains $R$.
Thus $\rho$ is equal to the radius of the disk circumscribed  on the triangle $abc$.
Denote by $o$ the center of this disk and by $p$ the middle of the arc $ab$.  
Clearly, the angle $\angle opa$ is right, the angle $\angle aop$ is equal $\frac{\pi}{3}$ and $|ap|=\frac{\Delta (R)}{2}$.
Therefore, if we look at the triangle $aop$, by (\ref{sin}) we obtain $\sin \frac{\Delta(R)}{2} = \sin \frac{\pi}{3}\sin \rho$. 
We easily evaluate $\rho$ which establishes the promised formula.
\end{proof}

\begin{lem}\label{equi}
The circumradius $\rho$ of a spherical equilateral triangle $T$ of thickness less than $\frac{\pi}{2}$ equals $\arctan \frac{\sqrt{9+8\tan^2\Delta(T)}-3}{2\tan \Delta (T)}$.
\end{lem}
\begin{proof}
Denote by $o$ the center of the disk circumscribed on our triangle, by $d$ the center of a side of this triangle and by $a$ an endpoint of this side.
Clearly $\angle doa  = \frac{\pi}{3}$, $|oa| = \rho$ and $|od|=\Delta (T) - \rho$.
Therefore by (\ref{cos})  we obtain
$\displaystyle\tan (\Delta (T) - \rho ) = \frac 12  \tan \rho $. 
Using the subtraction formula for the tangent function we can rewrite this equation as
$\displaystyle\frac{ \tan \Delta (T) - \tan \rho}{1+\tan \Delta (T) \tan \rho} = \frac{\tan \rho}{2}$.
This is equivalent to
$\displaystyle\tan \Delta (T) (\tan \rho ) ^2 +3 \tan \rho - 2\tan \Delta (T) =0$.
Consequently, by $\tan \rho >0$ we get  $\tan \rho = \frac{\sqrt{9+8\tan^2\Delta(T)}-3}{2\tan \Delta (T)}$, which ends the proof.
\end{proof}

\section{Covering a body of constant width over $\frac{\pi}{2}$ by a disk}

For any set $F$ on the sphere $S^2$ we define the set $F^\oplus$ as $\left\{ p: F\subset H(p)\right\}$.

\smallskip

In \cite{HN} and \cite{NS} there is used the notion of the polar set of the set $F$ on the sphere, which is defined as $F^o= \bigcap_{p\in F} H(p)$.

\begin{pro}
For every set $F$ on the sphere we have $F^o=F^\oplus$.
\end{pro}

\begin{proof}
For any point $q$ we have:
$\displaystyle q\in F^o\Leftrightarrow \forall_{p\in F}\  q\in H(p) \Leftrightarrow \forall_{p\in F}\ |pq| \le \frac{\pi}{2} \Leftrightarrow\forall_{p\in F}\ p\in H(q) \Leftrightarrow F \subset H(q) \Leftrightarrow q\in F^\oplus$.
\end{proof}

However, applying $F^\oplus$ is more convenient in this paper.

\medskip

We omit here a simple proof of the next lemma

\begin{lem}
If $C$ is a spherical convex body, then $C^\oplus$ is also a spherical convex body.
\end{lem}

By the way, observe that $(C^\oplus)^\oplus=C$ for a spherical convex body $C$.

\begin{pro}\label{pi-d}
If $W$ is a spherical convex body of constant width, then $W^\oplus$ is a spherical convex body of constant width $\pi - \Delta (W)$.
\end{pro}
\begin{proof}
Consider a hemisphere $H(a)$ supporting $W^\oplus$. 
We intend to show that $\textrm{width}_{H(a)}(W^\oplus)=\pi - \Delta (W)$.
Since clearly $a$ is a boundary point of $W$, by Theorem 7 of \cite{LMu} there exist hemispheres $K$ and $M$ supporting $W$ such that the lune $K\cap M$ is of thickness $\Delta(W)$ and $a$ is the center of $K/M$.
Denote the center of the semicircle $M/K$ by $b$ and the centers of the hemispheres $K$ and $M$ by $a'$ and $b'$, respectively.
Since $a,b,a',b'$ are in the same distance from both corners of $K\cap M$, all these points lay on the same great circle.
From the above we easily obtain that $|ab|+|a'b'|=|aa'|+|bb'|=\pi$ and thus $|a'b'|=\pi - \Delta (W)$.
Since every point of $W^\oplus$ is in the distance at most $\frac{\pi}{2}$ from $b$, the body $W^\oplus$ is contained in $H(b)$.
Therefore $W^\oplus$ is contained in the lune $H(a)\cap H(b)$, which is of thickness $|a'b'|$. It means that $\textrm{width}_{H(a)}(W^\oplus)$ is at most $ \pi - \Delta (W)$.

Assume that $\textrm{width}_{H(a)}(W^\oplus)< \pi - \Delta (W)$. 
Then there exists a point $\overline{b}$ such that the lune $H(a)\cap H(\overline{b})$ is of thickness less that $ \pi - \Delta (W)$ and contains $W^\oplus$.
From $|a\overline{b}| + \Delta (H(a)\cap H(\overline{b}))=\pi$ we conclude that $|a\overline{b}|>\Delta (W)$.
But this contradicts the fact that every body $W$ of constant width has diameter $\Delta (W)$.
Hence, $\textrm{width}_{H(a)}(W^\oplus)=\pi - \Delta (W)$, which ends the proof.
 
\end{proof}

In \cite{De} there was estimated the diameter of a spherical compact set by the function  of its circumradius (see: the second part of Theorem 2 in \cite{De}).
Recall this result on $S^2$ in a different form: if $d$ is the diameter of a  compact set and $\sigma$ is its circumradius, then  $\sin \sigma \le \frac{2\sqrt{3}}{3} \cdot \sin \frac{d}{2}$.
In particular it holds true for every spherical convex body $W$ of constant width and in this case $d=\Delta (W)$.
If $\Delta (W)$ is at most $\frac{2\pi}{3}$ then $\frac{2\sqrt{3}}{3} \cdot \sin \frac{d}{2}$ is at most $1$.
In this case our inequality is equivalent to the statement that every spherical convex body of constant width at most $\frac{2\pi}{3}$ can be covered by a disk of radius $\arcsin \left( \frac{2\sqrt{3}}{3} \cdot \sin \frac{\Delta(W)}{2}\right)$.
If $\Delta (W)$ is greater than $\frac{2\pi}{3}$ then $\frac{2\sqrt{3}}{3} \cdot \sin \frac{d}{2}$ is greater than $1$ and in this case our inequality does not estimate $\sigma$ in a non-trivial way.

According to Lemma \ref{Reul} the example of a spherical Reuleaux triangle shows that the estimate can not be improved for bodies of constant width at most $\frac{\pi}{2}$.
The following theorem describes the case of bodies of constant width at least $\frac{\pi}{2}$ and in particular gives an improvement of the estimate recalled  from \cite{De} for convex bodies of constant width greater than $\frac{\pi}{2}$.

\begin{thm}\label{const}
Every spherical body $W$ of constant width at least $\frac{\pi}{2}$ is contained in a disk of radius $\Delta(W) + \arcsin \left( \frac{2\sqrt{3}}{3} \cdot \cos \frac{\Delta(W)}{2}\right) - \frac{\pi}{2}$.\end{thm}
\begin{proof}
Observe that for $\Delta(W)=\frac{\pi}{2}$ we have 
$\displaystyle \arcsin \left( \frac{2\sqrt{3}}{3} \cdot \cos \frac{\Delta(W)}{2}\right)+\Delta (W) - \frac{\pi}{2}= \arcsin \left( \frac{2\sqrt{3}}{3} \cdot \sin \frac{\Delta(W)}{2}\right) $.
Therefore by the above recalled result of Dekster the thesis of the theorem holds true for bodies of constant width $\frac{\pi}{2}$.

Assume now that $\Delta (W)>\frac{\pi}{2}$.
By Proposition \ref{pi-d} the body $W^\oplus$ is of thickness $\pi - \Delta (W)$.
Hence by the above recalled result of Dekster, $W^\oplus$ is contained in a disk of radius $\displaystyle \arcsin \left( \frac{2\sqrt{3}}{3} \cdot \sin \frac{\pi-\Delta (W)}{2}\right)=\\ \arcsin \left( \frac{2\sqrt{3}}{3} \cdot \cos \frac{\Delta(W)}{2}\right) $.
Denote the center of this disk by $o$ and let $p$ be a boundary point of $W$.
By Proposition 3 of \cite{LMu} $W$ is smooth, and therefore there exists exactly one hemisphere supporting $W$ at $p$. 
Denote its center by $a$ and notice that $a$ is a boundary point of $W^\oplus$. 
Moreover, $H(p)$ is a supporting hemisphere of $W^\oplus$ and it supports $W^\oplus$ at $a$.
By the proof of the first part of Theorem 1 of \cite{L2}, there exists a unique point of $W^\oplus$ closest to $p$. 
Denote it by $b$. 
Again by Theorem 1 of \cite{L2} $b$ is the center of one of the two semicircles bounding the lune of thickness $\textrm{width}_{H(p)}(W^\oplus)= \Delta (W^\oplus)= \pi - \Delta (W)$.
Thus  $|ab|=\Delta (W^\oplus)$ and clearly $b\in ap$.
Hence $|op|\le |ob|+|bp|=|ob|+|ap|-|ab|=|ob|+ \frac{\pi}{2} - \left(\pi - \Delta (W)\right) = |ob|+\Delta (W) - \frac{\pi}{2}$.
Since $|ob|$ is at most $\arcsin \left( \frac{2\sqrt{3}}{3} \cdot \cos \frac{\Delta(W)}{2}\right)$, the distance between $o$ and $p$ is at most $\arcsin \left( \frac{2\sqrt{3}}{3} \cdot \cos \frac{\Delta(W)}{2}\right)+\Delta (W) - \frac{\pi}{2}$ which ends the proof.

\end{proof}

Observe that in general we can not improve the estimate from 
Theorem \ref{const}.
By proof of this theorem we see that the value $\Delta(W) + \arcsin \left( \frac{2\sqrt{3}}{3} \cdot \cos \frac{\Delta(W)}{2}\right) - \frac{\pi}{2}$ is attained for every $W$ such that $W^\oplus$ is a Reuleaux triangle.

\section{Covering reduced bodies of thickness at most $\frac{\pi}{2}$ by a disk}

The main theorem of this paper is analogous to Theorem of \cite{L3}. 
However, we are not able to present a similar proof as in \cite{L3} due to the lack of the notion parallelism on the sphere.
For this reason the proof of our main theorem is based on a different idea.

\begin{lem}\label{function}
Let $c$ be a positive number less than $\frac 14$ and $a$ be a number from 
 the interval $\left( \frac 12,\frac 12 + \sqrt{\frac 14 - c}\right)$. The function $f(x)= \sqrt{1-x} - \sqrt{1-x- \frac{c}{x}}$ satisfies $f(x) \le \max \left(  f\left( \frac 12 \right),f(a)\right)$ for every $x\in \left[ \frac 12, a\right]$.
\end{lem}
\begin{proof}
Notice that if $x>0$, then  $1-x- \frac{c}{x}\ge 0$ is equivalent to $x^2-x+c\le 0$.
This inequality is satisfied for $x\in \left[ \frac 12 - \sqrt{\frac 14 - c}, \frac 12+\sqrt{\frac 14 - c}\right]$. 
In particular we conclude that $\sqrt{1-x- \frac{c}{x}}=0$ for $x= \frac 12 + \sqrt{\frac 14 - c}$ and $f(x)$ is well defined in the interval $ \left( \frac 12,\frac 12 + \sqrt{\frac 14 - c}\right)$ 

In order to prove the thesis we check the sign 
of the first derivative of $f(x)$. 
We have
$\displaystyle f'(x) = - \frac{1}{2\sqrt{1-x}} - \frac{1}{2\sqrt{1-x- \frac{c}{x}}}\cdot \left( -1 + \frac{c}{x^2}\right) =\frac{1}{2\sqrt{1-x- \frac{c}{x}}}\cdot \left(1- \frac{c}{x^2} - \sqrt{1- \frac{c}{x(1-x)}} \right) $.

Put $g(x)= 1- \frac{c}{x^2} - \sqrt{1- \frac{c}{x(1-x)}}$.
Clearly,  for every $x$ the sign of $f'(x)$ is the same as the sign of $g(x)$.
We have $\displaystyle g(x)=0 \Leftrightarrow 1- \frac{c}{x^2} = \sqrt{1- \frac{c}{x(1-x)}} \Leftrightarrow 1- \frac{2c}{x^2} + \frac{c^2}{x^4}=1- \frac{c}{x(1-x)} \Leftrightarrow \frac{1}{1-x}-\frac 2x + \frac{c}{x^3}=0 \Leftrightarrow 3x^3-2x^2+c(1-x)=0$. 

For $V(x)= 3x^3-2x^2+c(1-x)$ we have $V(0)=c>0$ and $V\left(\frac 12\right) = \frac 12 \left( c-\frac 14\right)<0$.
Thus $V(x)$ has three zeros: one is less than $0$, one is in the interval $\left( 0, \frac 12\right)$, and one is greater than $\frac 12$.
Hence, $g(x)$ has exactly one zero in the interval $\left( \frac 12, \frac 12 + \sqrt{\frac 14 - c}\right)$.
Denote it by $x_0$.
Due to the continuity of the function $g(x)$ in the interval $\left( \frac 12, \frac 12 + \sqrt{\frac 14 - c}\right)$, it has constant sign in the interval $\left( \frac 12,x_0\right) $ and constant sign in the interval $\left(  x_0,\frac 12 + \sqrt{\frac 14 - c} \right)$.
Notice that
$\displaystyle g\left( \frac 12\right) = 1-4c -\sqrt{1-4c} = (1-4c)(\sqrt{1-4c}-1)<0$
and 
$\displaystyle g\left(\frac 12 + \sqrt{\frac 14 - c}\right) = \sqrt{\frac 12 - \sqrt{\frac 14 - c}}>0$.
Hence, $g(x)<0$ for $x\in \left( \frac 12, x_0\right)$ and $g(x)>0$ for $x\in \left( x_0,a \right)$.
Therefore $f(x)$ is decreasing in $\left( \frac 12, x_0\right)$ and increasing in $\left( x_0,a \right)$. 
The thesis of our lemma is an immediate consequence of this statement.

\end{proof}
\begin{lem}\label{angle}
Let $a,b$ be points on a spherical circle. 
Consider any point $t$ such that $a,t, b$ lay on this circle in this order according to the positive orientation. 
The measure of the angle $\angle atb$  is the greatest if $t$ is equidistant
from $a$ and $b$.
\end{lem}
\begin{proof}
Denote by $o$ the center of our circle and by $\rho$ its radius.
Denote by $t'$ the point on our circle laying in equal distances from $a$ and $b$, and on the same side of the great circle containing $ab$ as $t$.
Put $\alpha = \frac 12|\angle aot|$] and $\beta =\frac 12|\angle bot| =$, and let $k$ be the midpoint of the arc $at$.
Observe that $|\angle aot'|=|\angle bot'|= \alpha +\beta$.
Since $|\angle tok|=\alpha$, applying (\ref{cot}) we easily obtain that $|\angle ota| =  \textrm{arccot} ( \cos \rho \tan \alpha )$.
Analogously $|\angle otb| = \textrm{arccot} ( \cos \rho \tan \beta )$ and $|\angle ot'a|  =|\angle ot'b| = \textrm{arccot} \left( \cos \rho \tan \frac{\alpha +\beta}{2} \right)$.
Our aim is to show that  $|\angle at'b|\ge |\angle atb|$.

This inequality is equivalent to
$\displaystyle\textrm{arccot} \left( \cos \rho \tan \frac{\alpha +\beta}{2}\right)\ge\\ \frac{\textrm{arccot} ( \cos \rho \tan \alpha )+\textrm{arccot} ( \cos \rho \tan \beta )}{2}$.
In order to show this it is sufficient to show that the function $f(x) =
\textrm{arccot} ( \cos \rho \tan x )$ is concave in the interval $ \left( 0, \frac{\pi}{2}\right)$.
The reader may check that
$\displaystyle f''(x) = \frac{2\cos\rho \sin x \cos x (\cos^2\rho- 1)}{(\cos^2 x + \cos^2\rho \sin^2 x)^2}$.

It is easily seen that $f''(x)<0$ for $x\in \left( 0, \frac{\pi}{2}\right)$  and therefore the function $f(x)$ is concave in the interval $\left( 0, \frac{\pi}{2}\right)$ which completes the proof. 
\end{proof}

\begin{thm}\label{second}
Every reduced spherical body $R$ of thickness at most $\frac{\pi}{2}$ is contained in a disk of radius $\rho = \arctan \left( \sqrt{2} \cdot \tan \frac{\Delta(R)}{2}\right)$.
\end{thm}
\begin{proof}
%Helly's Theorem
Note that every boundary point of $R$ belongs to an arc whose ends are extreme points of $R$.
Therefore if a disk contains all extreme points of $R$, then it contains all boundary points of $R$ and so the whole body $R$. 
Thus, it is sufficient to show that all extreme points of $R$ are in a disk of radius $\rho$.
Moreover, according to the spherical Helly's Theorem (see: \cite{Mo} and \cite{DGK}), it is even sufficient to show that every three extreme points of $R$ are contained in a disk of radius $\rho$.

Let  $e_1,e_2,e_3$ be any three different extreme points of $R$.
By Theorem 4 of \cite{L2}  there exist lunes $L_1, L_2, L_3$ such that $e_i$ is the center of one of the semicircles bounding $L_i$ for $i=1,2,3$. 
Denote by $f_i$ the center of the other semicircle bounding $L_i$ for $i=1,2,3$. 

\smallskip

% Case 1
First let us consider the case when two points from amongst $e_1,e_2,e_3$, say $e_1$ and $e_2$, lay in $L_3$ on the same side of the great circle containing $e_3$ and $f_3$. 
Since $e_1f_1$ and $e_3f_3$ intersect 
(see the proof of Lemma 2 in \cite{LMu}), the distance from $e_1$ to $e_3f_3$ is less or equal than the distance from $e_1$ to the point of intersection of $e_1f_1$ with $e_3f_3$, and so it is less or equal $\Delta (R)$.
For the same reason the distance from $e_2$ to $e_3f_3$ is at most $\Delta (R)$. 
Denote by $c$ the corner of $L_3$ laying on the same side of the great circle containing $e_3$ and $f_3$ as $e_1$ and $e_2$.
Denote by $k$ the point of $e_3c$ such that $|e_3k|=\Delta (R)$ and by $l$ the point of $f_3c$ such that $|f_3l|=\Delta (R)$.
Since $e_1,e_2,e_3\in \textrm{conv}\left\{e_3,f_3,k,l\right\}$, it is sufficient to show that $\textrm{conv}\left\{e_3,f_3,k,l\right\}$ may be covered by a disk of radius $\rho$.
The triangle $e_3f_3k$ is contained in a quarter of a disk of thickness $\Delta(R)$, so by Lemma \ref{quart} it can be covered by a disk of radius $\rho$.
What is more such a disk is unique thanks to Lemma $\ref{quart}$
For the same reason the triangle $e_3f_3l$ can be covered by a disk of radius $\rho$ and such a disk is unique.
It is easily seen that the disk is the same for both triangles. 
This disk covers $\textrm{conv}\left\{e_3,f_3,k,l\right\}$, which ends proof in this case.

\smallskip

%Case 2
 Assume now that for any $i,j,k$ such that $\{i,j,k\}= \{1,2,3\}$ the points $e_i$ and $e_j$ lay in $L_k$ on the different sides of the great circle containing $e_k$ and $f_k$.
Let us stay with this assumption up to the end of the proof.

\smallskip

%Case 2.1
Let us consider the case when the triangle $e_1e_2e_3$ is obtuse or right. 
Without losing the generality we can assume that the angle $\angle e_1e_3e_2$ is obtuse or right.
Let $k$ be the point on the same side of the great circle containing $e_1$ and $e_2$ as $e_3$, such that $|e_1k|=|e_2k|$ and the angle $\angle e_1ke_2$ is right.
By  Theorem 8 of \cite{LMu} $|e_1e_2|$ is at least $\arccos (\cos^2\Delta (R))$.
Therefore $|e_1k|$ and $|e_2k|$ are at most $\Delta (R)$, because otherwise by (\ref{Pyth}) the lenght $|e_1e_2|$ is greater that $\arccos (\cos^2\Delta (R))$.
Thus by Lemma \ref{quart} the circumradius of the triangle $e_1ke_2$ is at most $\arctan \left( \sqrt{2} \cdot \tan \frac{\Delta(R)}{2}\right)$.

By Lemma \ref{angle} we easily obtain that for any point $t$ from the circle circumscribed on $e_1ke_2$ laying on the same side of the great circle containing $e_1$ and $e_2$ as $k$, the angle $\angle e_1te_2$ is at least $\frac{\pi}{2}$.
Thus since $\angle e_1e_3e_2$ is obtuse or right, $e_3$ must lay inside this circumscribed circle, which ends the proof in this case.

\smallskip

%Case 2.2
The last case is when the triangle $e_1e_2e_3$ is  acute. 
Clearly, all heights of this triangle have lengths less that $\Delta (R)$.
Let $g$ be the point closest to $e_2$ such that $e_2$ is on the arc $e_1g$ and $g$ satisfies at least one of following conditions: $\angle e_1e_3g$ is right or $\textrm{dist}(g, e_1e_3)=\Delta (R)$ or $\textrm{dist}(e_1, ge_3)=\Delta (R)$.
%Case 2.2.1
If $\angle e_1e_3g$ is right, then $|e_1e_3|$ and $|e_3g|$ are at most $\Delta (R)$. 
Therefore the triangle $e_1ge_3$ is contained in a right triangle with legs of length $\Delta (R)$, and so in the quarter of a disk of thickness $\Delta (R)$.
Hence, by Lemma \ref{quart}, $e_1,e_2,e_3$ can be covered by a disk of radius $\rho$.

%Case 2.2.2
Otherwise, the angle $\angle e_1e_3g$ is acute. 
Since $|\angle e_1ge_3|<|\angle e_1e_2e_3|$, the triangle $e_1ge_3$ is acute.
One height of this triangle, say this from the vertex $g$, is of length $\Delta (R)$.
Moreover, the height of this triangle from vertex $e_1$ is of length at most $\Delta (R)$.
Let $j$ be the point closest to $e_3$ such that $e_3$ is on the arc $e_1j$ and $j$ satisfies at least one of following conditions: $\angle e_1gj$ is right or $\textrm{dist}(j, e_1g)=\Delta (R)$ or $\textrm{dist}(e_1, gj)=\Delta (R)$.
%Case 2.3
If $\angle e_1gj$ is right (which by the way is possible only if $\Delta (R) =\frac{\pi}{2}$), then we apply the same argument as in 
the preceding case.
Otherwise the triangle $e_1gj$ is acute and has two heights of length $\Delta (R)$, say from vertices $g,j$.
Observe that from the construction of this triangle we obtain that the third height is of length at most $\Delta (R)$.
Since the points $e_1,e_2,e_3$ are contained in the triangle $e_1gj$, it remains to prove that this triangle may be covered by 
a disk of radius $\rho$.

\smallskip

Observe that since the triangle $e_1gj$ has two equal heights, it is an isosceles triangle.
Denote the angle at $e_1$ by $2\alpha$, the center of the circle circumscribed on $e_1gj$ by $o$ and the radius of this circle by $\sigma$.
Denote also the point at the middle of $je_1$ by $k$ and the point on $e_1g$ closest to $j$ by $h$.
Clearly, the triangles $hje_1$ and $oke_1$ are right.
Put $B=|je_1|$ and notice that $|ke_1|=\frac B2$.
Since $e_1gj$ is isosceles, we have $|\angle oe_1k|= \alpha$.
The length of the arc $hj$ is $\Delta (R)$.

From (\ref{cos}) for the triangle $oke_1$ we obtain $\tan \frac B2 =\cos \alpha \tan \sigma$. 
Formula (\ref{sin}) for the triangle $hje_1$ gives $\sin \Delta (R) = \sin 2\alpha \sin B$ and in different form 
$\sin B\cos \alpha = \frac{\sin \Delta (R)}{2\sin\alpha}$.
Using formulas from last two sentences and the trigonometric formula 
 $\tan \frac B2 = \frac{1-\cos B}{\sin B}$, we obtain
$\displaystyle\tan \sigma = \frac{ \tan \frac B2}{\cos \alpha} = \frac{1-\cos B}{\sin B\cos \alpha}= \frac{1-\cos B}{\frac{\sin\Delta (R)}{2\sin \alpha}} = 
\frac{2\sin\alpha \left( 1 - \sqrt{1-\sin^2B}\right)}{\sin\Delta (R)}= \frac{2\sin\alpha \left( 1 - \sqrt{1-\frac{\sin^2\Delta (R)}{\sin^22\alpha}}\right)}{\sin \Delta (R)}=
\frac{2 \left( \sin\alpha - \sqrt{\sin^2\alpha-\frac{\sin^2\alpha\sin^2\Delta (R)}{\sin^22\alpha}}\right)}{\sin \Delta (R)}  =
\frac{2 \left( \sqrt{1-\cos^2\alpha} - \sqrt{1- \cos^2\alpha-\frac{\sin^2\Delta (R)}{4\cos^2\alpha}}\right)}{\sin \Delta (R)}$.
The greatest possible value of $\alpha$ is $\frac{\pi}{4}$ and the smallest is this value for which the third height is of length $\Delta(R)$.
In the first case the triangle $e_1jg$ is right and by Lemma \ref{quart} we have $\tan \sigma_1 = \sqrt{2} \cdot \tan \frac{\Delta(R)}{2}$. 
By the formula $\tan \Delta (R) = \frac{2 \tan \frac{\Delta (R)}{2}}{1- \tan^2\frac {\Delta (R)}{2}}$ we easily obtain that in this case
\[\tan \sigma_1 = \sqrt{2} \cdot \frac{\sqrt{4+4\tan^2 \Delta (R)}-2}{2\tan\Delta (R)} =\frac{\sqrt{2+2\tan^2 \Delta (R)}-\sqrt{2}}{\tan\Delta (R)} \]
In the second case the triangle $e_1jg$ is equilateral and by Lemma \ref{equi} in this case $\tan \sigma _2= \frac{\sqrt{9+8\tan^2\Delta(R)}-3}{2\tan \Delta (R)}$.
For shortness denote $\tan \Delta (R)=t$. 
We have
$\displaystyle\tan\sigma_2 = \frac{\sqrt{9+8t^2}-3}{2t} = \frac{1}{2t}\left(\sqrt{9+8t^2}-\sqrt{8+8t^2} -3+ \sqrt{8+8t^2}\right) =\frac{1}{2t}\left(\frac{1}{\sqrt{9+8t^2}+\sqrt{8+8t^2}} -3+ \sqrt{8+8t^2}\right) <\frac{1}{2t}\left(\frac{1}{\sqrt{9}+\sqrt{8}} -3+ \sqrt{8+8t^2}\right)= \frac{1}{2t}\left(3- 2\sqrt{2} -3+ \sqrt{8+8t^2}\right) = \frac{\sqrt{2+2t^2}-\sqrt{2}}{t}=\tan \sigma_1$.

If we put $\cos^2\alpha =x$ and $\frac{\sin^2\Delta (R)}{4}=c$, then by Lemma \ref{function} we conclude that $\tan\sigma$ has the greatest value $\max \left( \tan \sigma_1, \tan \sigma_2 \right) = \sqrt{2} \cdot \tan \frac{\Delta(R)}{2}$, from which we obtain the thesis of theorem.  

\end{proof}

Note that by Lemma \ref{quart} in general we can not improve the estimate from Theorem \ref{second}.

\end{document}